\documentclass[12pt]{article}

\pdfoutput=1

\usepackage[a4paper]{geometry}
\usepackage{amsthm}
\usepackage{amsmath}
\usepackage{amssymb}
\usepackage{amsfonts}
\usepackage{graphicx}
\usepackage{enumerate}
\usepackage[colorlinks=true,citecolor=black,linkcolor=black,urlcolor=blue]{hyperref}
\newcommand{\doi}[1]{\href{http://dx.doi.org/#1}{\texttt{doi:#1}}}

\theoremstyle{plain}
\newtheorem{theorem}{Theorem}
\newtheorem{lemma}[theorem]{Lemma}

\newtheorem{proposition}[theorem]{Proposition}

\theoremstyle{definition}
\newtheorem{definition}[theorem]{Definition}

\theoremstyle{remark}

\title{Matroids that classify forests }

\author{
Lorenzo Traldi\\
\small Lafayette College\\[-0.8ex]
\small Easton, Pennsylvania, U.S.A. 18042\\
\small\tt traldil@lafayette.edu
}

\date{}

\begin{document}

\maketitle

\begin{abstract}
Elementary arguments show that a tree or forest is determined (up to isomorphism) by binary matroids defined using the adjacency matrix.

\bigskip\noindent \textbf{Keywords:} adjacency matrix, forest, local equivalence, matroid, pivot, tree
\end{abstract}

\section{Introduction}

Let $G$ be a simple graph with $|V(G)|=n$. The adjacency matrix of $G$ is an $n \times n$ matrix $A(G)$, whose entries lie in the two-element field $GF(2)$. We consider two matroids introduced in \cite{Tnewnew}, which are defined using $A(G)$.

\begin{definition}
The \emph{restricted isotropic matroid} of $G$ is the binary matroid represented by the $n \times 2n$ matrix
\[
IA(G)=\begin{pmatrix}
I & A(G)
\end{pmatrix} \text ,
\]
where $I$ is the identity matrix. The matroid is denoted $M[IA(G)]$.
\end{definition}

\begin{definition}
The \emph{isotropic matroid} of $G$ is the binary matroid represented by the $n \times 3n$ matrix 
\[
IAS(G)=\begin{pmatrix}
I & A(G) & I+A(G)
\end{pmatrix} \text ,
\]
where $I$ is the identity matrix. The matroid is denoted $M[IAS(G)]$. 
\end{definition}

As explained in \cite{BT, Tnewnew}, the names of these matroids reflect their connection with Bouchet's theories of delta-matroids and isotropic systems \cite{Bi1,bouchet1987,Bi2, B1}. It is not the purpose of the present paper to give an account of this connection. Rather, our purpose is to present a direct, elementary proof of the following.

\begin{theorem} 
\label{main}
Isotropic matroids and restricted isotropic matroids are classifying invariants for forests. That is, if $F$ and $F'$ are forests then any one of these three statements implies the other two:
\begin{enumerate}
    \item The forests $F$ and $F'$ are isomorphic.
    \item The restricted isotropic matroids $M[IA(F)]$ and $M[IA(F')]$ are isomorphic.
    \item The isotropic matroids $M[IAS(F)]$ and $M[IAS(F')]$ are isomorphic.
\end{enumerate}
\end{theorem}

The implications $1 \implies 2$ and $1 \implies 3$ of Theorem \ref{main} are obvious. The converses may be deduced from more general results that have appeared in the literature, as follows. Oum \cite[Cor.\ 3.5]{Oum} showed that for two bipartite graphs, a condition equivalent to isomorphism of the restricted isotropic matroids implies that the graphs are related through pivoting, and hence also related through local complementation. Bouchet \cite{Btree} used the split decomposition of Cunningham \cite{Cu} to show that forests related through local complementation are isomorphic. The implication $2 \implies 1$ follows from these two results. The implication $3\implies 1$ follows from the same result of Bouchet \cite{Btree}, using the fact that graphs with isomorphic isotropic matroids are related through local complementation \cite{Tnewnew}. A different way to deduce $3 \implies 1$ from the results of \cite{Tnewnew} was given in \cite{BT}.

In the following pages we provide a self-contained exposition of Theorem \ref{main}, which requires only elementary ideas of linear algebra and matroid theory. Before proceeding, we should thank Robert Brijder for his long collaboration and friendship, which have provided many instances of inspiration and understanding.

\section{Binary matroids}
\label{secbin}

In this section we briefly recall some fundamental elements of the theory of binary matroids. We refer the reader to Oxley \cite{O} for a thorough discussion of matroid theory.

Let $B$ be a matrix with entries in $GF(2)$, the field of two elements. The \emph{binary matroid} $M[B]$ represented by $B$ has an element for each column of $B$. A nonempty subset $D \subseteq M[B]$ is \emph{dependent} if the corresponding columns of $B$ are linearly dependent; a minimal dependent set is a \emph{circuit} of $M[B]$. An element $x$ of $M[B]$ is a \emph{loop} if $\{x\}$ is dependent; that is, every entry of the corresponding column of $B$ is $0$. An element $x$ of $M[B]$ is a \emph{coloop} if no circuit includes $x$; that is, the corresponding column of $B$ is not included in the linear span of the other columns. Two elements $x$ and $y$ are \emph{parallel} if $\{x,y\}$ is a circuit; that is, the corresponding columns of $B$ are nonzero and identical. A \emph{basis} of $M[B]$ is a subset that corresponds to a basis of the column space of $B$.

If $B'$ is another matrix with entries in $GF(2)$, then a bijection between the matroids $M[B]$ and $M[B']$ is an \emph{isomorphism} iff it matches dependent sets to dependent sets. In general, there will be many different matrices representing matroids isomorphic to $M[B]$. Here is a simple way to produce some of them.

\begin{lemma}
\label{iso}Suppose a particular entry of $B$, say $b_{ij}$, is equal to $1$. Let $\kappa$ be any column vector of the same size as the columns of $B$, whose $i$th entry is $0$. Let $B'$ be the matrix obtained from $B$ by adding $\kappa$ to each column of $B$ that has a $1$ in the $i$th row. Then $M[B] \cong M[B']$.
\end{lemma}
\begin{proof}
A set of columns of $B$ sums to $0$ if and only if the corresponding set of columns of $B'$ sums to $0$.
\end{proof}

If $S \subseteq M[B]$ then $S$ is made into a \emph{submatroid} of $M[B]$ by defining a subset of $S$ to be dependent in $S$ if and only if it is dependent in $M[B]$. The submatroid is also said to be obtained from $M[B]$ by \emph{deleting} the subset $M[B] \setminus S$. The submatroid is denoted $M[B] | S$ or $M[B] \setminus(M[B] \setminus S)$. If $B_S$ is the submatrix of $B$ consisting of columns corresponding to elements of $S$ then it is easy to see that $M[B_S] = M[B] | S$. 

Circuits are used to define an equivalence relation $\sim$ on $M[B]$, as follows: every element $x$ has $x \sim x$; if $x$ and $y$ are two elements of one circuit then $x \sim y$; and if $x \sim y \sim z$ then $x \sim z$. The equivalence classes of $\sim$ are the \emph{components} of $M[B]$. Components of cardinality $1$ are singletons $\{x\}$, where $x$ is a loop or coloop. As the circuits of a submatroid of $M[B]$ are also circuits of $M[B]$ itself, every component of a submatroid of $M[B]$ is contained in a component of $M[B]$. In general, the components are not the same; a component of $M[B]$ may contain several different components of a given submatroid. In particular, a coloop of a submatroid need not be a coloop of $M[B]$.

If $x \in M[B]$ then the \emph{contraction} $M[B]/x$ is a matroid on the set $M[B] \setminus \{x\}$. A subset $D \subseteq M[B] \setminus \{x\}$ is defined to be dependent in $M[B]/x$ if and only if $D \cup \{x\}$ is dependent in $M[B]$. It is not immediately evident, but it does turn out that $M[B]/x$ is represented by a matrix. To see why, notice first that we may presume that the column of $B$ representing $x$ has no more than one nonzero entry. If $x$ is a loop, then of course the entries in the $x$ column are all $0$. Otherwise, if $b_{ij}$ is a nonzero entry of the column representing $x$, and $\kappa$ is the column vector obtained from the $j$th column of $B$ by changing the $i$th entry to $0$, then according to Lemma \ref{iso} we may add $\kappa$ to every column of $B$ with a $1$ in the $i$th row, without changing the matroid $M[B]$. The effect on the $j$th column is to change every entry other than $b_{ij}$ to $0$.

If $x$ is a loop, then a matrix representing $M[B]/x$ is obtained by removing the $x$ column of $B$. If $x$ is not a loop, the following proposition applies.

\begin{proposition}
\label{contr}
Suppose $x$ corresponds to the $j$th column of $B$, and $b_{ij}$ is the only nonzero entry of this column. Let $B'$ be the submatrix of $B$ obtained by removing both the $i$th row and the $j$th column of $B$. Then $M[B'] \cong M[B]/x$.
\end{proposition}
\begin{proof}
If $S$ is a set of columns of $B'$ then the sum of $S$ is $0$ if and only if the sum of the corresponding columns of $B$ is equal to either $0$ or the $j$th column of $B$. Either way, we get a dependent set of columns of $B$ by including the $j$th column along with the columns corresponding to elements of $S$.
\end{proof}

Suppose $B$ is an $m \times n$ matrix with entries in $GF(2)$, of rank $r$. If $r=0$ or $r=n$ then the \emph{standard representation} of $M[B]$ is $M[0_n]$ or $M[I_n]$, where $0_n$ and $I_n$ are the $n \times n$ zero and identity matrices. If $r \in \{1, \dots, n-1 \}$, choose $r$ columns of $B$ that constitute a basis of the column space of $B$. Permuting columns, we may presume these $r$ columns are the first $r$ columns of $B$. Then the linear relations that hold among the columns of $B$ are precisely the same as the linear relations that hold among the columns of the matrix $\begin{pmatrix} I_r & A \end{pmatrix}$, where $I_r$ is the $r \times r$ identity matrix and $A$ is the matrix whose columns record the coefficients in formulas for the $(r+1)$st, $\dots, n$th columns of $B$ as linear combinations of the first $r$ columns. As the columns of $B$ and $\begin{pmatrix} I_r & A \end{pmatrix}$ satisfy the same linear relations, it must be that $M[B] \cong M[\begin{pmatrix} I_r & A \end{pmatrix}]$. A matrix of the form $\begin{pmatrix} I_r & A \end{pmatrix}$ is a \emph{standard representation} of the matroid $M[B]$.

The \emph{dual} of $M[B]$ is a matroid on the same ground set as $M[B]$, whose bases are the complements of bases of $M[B]$. The dual of $M[B]$ is denoted $M[B]^*$. The duals of binary matroids are themselves binary matroids:

\begin{proposition}
\label{duals}
Given an $r \times n$ standard representation $M[\begin{pmatrix} I_r & A \end{pmatrix}]$ of a binary matroid $M$ of rank $r \in \{1, \dots, n-1 \}$, the dual matroid $M^*$ is the binary matroid $M[\begin{pmatrix} A^T & I_{n-r} \end{pmatrix}]$, where $A^T$ is the transpose of $A$.
\end{proposition}
\begin{proof}
The proposition is equivalent to the claim that for any subset $S \subseteq\{1, \dots, n\}$ of size $r$, the columns of $\begin{pmatrix} I_r & A \end{pmatrix}$ with indices from $S$ are linearly independent if and only if the columns of $\begin{pmatrix} A^T & I_{n-r} \end{pmatrix}$ with indices not from $S$ are linearly independent. 

Suppose $S$ intersects both $\{1, \dots, r\}$ and $\{r+1, \dots, n\}$. Then after permuting the rows and columns of $\begin{pmatrix} I_r & A \end{pmatrix}$, we may assume that $S = \{s+1, \dots, s+r \}$ where $1 \leq s \leq r-1$. Let the submatrix of $\begin{pmatrix} I_r & A \end{pmatrix}$ with column indices from $S$ be
\[
\begin{pmatrix} 0 & A_S \\ I_{r-s} & A'  \end{pmatrix}.
\]
The columns of this matrix are independent if and only if the $s \times s$ submatrix $A_S$ is nonsingular. If $s+r=n$ then the columns of $\begin{pmatrix} A^T & I_{n-r} \end{pmatrix}$ with indices not from $S$ constitute the submatrix $A_S^T$, which is nonsingular if and only if $A_S$ is nonsingular. 

If $s+r<n$, let
\[
\begin{pmatrix} I_r & A \end{pmatrix} = \begin{pmatrix} I_s & 0 & A_S & A'' \\ 0 & I_{r-s} & A' & A''' \end{pmatrix}.
\]
Then
\[
\begin{pmatrix} A^T & I_{n-r} \end{pmatrix} = \begin{pmatrix} A^T_S & (A')^T & I_s & 0\\ (A'')^T & (A''')^T & 0 & I_{n-r-s}  \end{pmatrix} \text{,}
\]
so the columns of $\begin{pmatrix} A^T & I_{n-r} \end{pmatrix}$ with indices not from $S$ constitute the submatrix
\[
\begin{pmatrix}  A^T_S & 0 \\ (A'')^T & I_{n-r-s} \end{pmatrix} \text{,}
\]
which is nonsingular if and only if $A_S$ is nonsingular.

If $S$ does not intersect $\{1, \dots, r\}$ then a similar argument applies, without $A'$ and $A''$. If $S$ does not intersect $\{r+1, \dots, n\}$ then the columns of $\begin{pmatrix} I_r & A \end{pmatrix}$ with indices from $S$ are the columns of $I_r$, and the columns $\begin{pmatrix} A^T & I_{n-r} \end{pmatrix}$ with indices not from $S$ are the columns of $I_{n-r}$. \end{proof}

We emphasize that a matroid and its dual are defined on the same ground set. When two matroids on different ground sets are said to be duals of each other, the relationship between them must involve a specified bijection between their ground sets.

We should also mention that the special cases $r=0$ and $r=n$ are excluded from Proposition \ref{duals}; this is done merely for consistency of notation. When $r=0$, the matroid $M=M[0_n]$ has the dual $M[I_n]$. And when $r=n$, $M=M[I_n]$ has the dual $M[0_n]$.

\section{Isotropic matroids of forests}
\label{secmat}

If $F$ is a forest then we use the following notation for the columns of the matrices $IA(F)$ and $IAS(F)$, and for the corresponding elements of the matroids $M[IA(F)]$ and $M[IAS(F)]$. If $v \in V(F)$ the $v$ column of $A(F)$ is denoted $\chi(v)$, or $\chi_F(v)$ if it is necessary to specify $F$. The column of the identity matrix $I$ with a $1$ in the row of $IA(F)$ that contains the $v$ row of $A(F)$ is denoted $\phi(v)$, or $\phi_F(v)$. The column of $I+A(F)$ corresponding to $v$ is denoted $\psi(v)$, or $\psi_F(v)$.

Some simple observations are given in Proposition \ref{props} below. Recall that a leaf in a forest is a vertex with precisely one neighbor. For clarity, we refer to the connected components of a forest $F$, and the components of the matroids $M[IA(F)]$ and $M[IAS(F)]$. 

\begin{proposition}
\label{props}
If $F$ is a forest, $M[IA(F)]$ and $M[IAS(F)]$ have the following properties.
\begin{enumerate}
    \item $M[IA(F)]$ is a submatroid of $M[IAS(F)]$.
    \item If $v$ is isolated, then $\chi(v)$ is a loop. There are no other loops in $M[IAS(F)]$.
    \item If $v$ is isolated, then $\phi(v)$ is a coloop in $M[IA(F)]$. Also $\phi(v)$ and $\psi(v)$ are parallel in $M[IAS(F)]$, and $\{\phi(v),\psi(v)\}$ is a component of $M[IAS(F)]$.
    \item If $v$ is a leaf adjacent to $w$, then $\chi(v)$ and $\phi(w)$ are parallel.
    \item If $v_1$ and $v_2$ are two leaves adjacent to $w$, then $\chi(v_1)$ and $\chi(v_2)$ are parallel.
    \item If a connected component of $F$ has precisely two vertices $v$ and $w$, then $\psi(v)$ and $\psi(w)$ are parallel in $M[IAS(F)]$.
    \item All parallels in $M[IAS(F)]$ fall under one of the above descriptions. 
\end{enumerate}
\end{proposition}
\begin{proof}
Only item 7 is not obvious. Suppose $x$ and $y$ are parallel in $M[IAS(F)]$. Then the $x$ and $y$ columns of $IAS(F)$ are nonzero and identical. 

If the two columns have more than one nonzero entry, they are not $\phi$ columns. If they are both $\chi$ columns, the two corresponding vertices of $F$ share two neighbors; this is impossible in a forest. If one is a $\chi$ column and the other is a $\psi$ column then the corresponding vertices of $F$ are neighbors, with a shared neighbor; again, this is impossible. If both columns are $\psi$ columns and they have more than two nonzero entries, then again, the two corresponding vertices are neighbors with a shared neighbor, an impossibility. If both columns are $\psi$ columns and they have only two nonzero entries, then they fall under item 6.

If the two columns have only one nonzero entry, then they cannot both be $\phi$ columns, and it cannot be that one is a $\chi$ column and the other is a $\psi$ column. All the other $\phi,\chi,\psi$ combinations are covered in items 3 -- 6.
\end{proof}

Here are two obvious lemmas.

\begin{lemma}
\label{trips}
If $v \in V(F)$ then $\{\phi(v),\chi(v),\psi(v)\}$ is a dependent set in $M[IAS(F)]$. We call it the \emph{vertex triple} of $v$. It is a circuit of $M[IAS(F)]$ if and only if $v$ is not an isolated vertex.
\end{lemma}

\begin{lemma}
\label{ncircs}
If $v \in V(F)$ then $\{\chi(v)\} \cup \{\phi(w) \mid vw \in E(F) \}$ is a circuit in $M[IA(F)]$. We call it the \emph{neighborhood circuit} of $v$.
\end{lemma}
Neighborhood circuits and vertex triples determine the components of $M[IA(F)]$ and $M[IAS(F)]$.
\begin{proposition}
\label{comps}
For each connected component $C$ of $F$, the matroid $M[IA(F)]$ has two components. For each $v \in V(F)$, each of the matroid components contains precisely one of $\phi(v),\chi(v)$. If the two matroid components are thought of as matroids on the set $V(C)$, they are duals of each other. On the other hand, $M[IAS(F)]$ has two components for each isolated vertex of $F$, and one component for each connected component of $F$ with more than one vertex.
\end{proposition}
\begin{proof}
Let $v \in V(F)$. If $v$ is isolated then items 2 and 3 of Proposition \ref{props} tell us that the vertex triple of $v$ is the union of two components of $M[IAS(F)]$, and $\{\phi(v),\chi(v)\}$ is the union of two components of $M[IA(F)]$. These two components are duals because $\phi(v)$ is a coloop of $M[IA(F)]$, and $\chi(v)$ is a loop.

Suppose $v$ is not isolated. The vertex triple of $v$ and the neighborhood circuit of $v$ are both circuits of $M[IAS(F)]$, so one component of $M[IAS(F)]$ contains the vertex triple of $v$ and also contains $\phi(w)$ for every neighbor $w$ of $v$. Iterating this observation, we conclude that one component of $M[IAS(F)]$ contains the vertex triple of $w$ for every vertex $w$ in the connected component of $F$ that contains $v$. 

In contrast, vertex triples are not circuits of $M[IA(F)]$. Using the neighborhood circuits, we see that there are two components of $M[IA(F)]$ associated with $v$. One component includes 
\[
\{ \phi(v) \} \cup \{ \chi(w) \mid d(v,w) \text{ is finite and odd} \} \cup \{ \phi(w) \mid d(v,w) \text{ is finite and even} \} 
\]
and the other component includes
\[
\{ \chi(v) \} \cup \{ \phi(w) \mid d(v,w) \text{ is finite and odd} \} \cup \{ \chi(w) \mid d(v,w) \text{ is finite and even} \} .
\]
Here $d(v,w)$ is the distance from $v$ to $w$ in $F$, i.e. the number of edges in a path from $v$ to $w$, if such a path exists, and $\infty$ otherwise. The two matroid components are naturally represented by matrices of the form $\begin{pmatrix} I_r & A_1 \end{pmatrix}$ and $\begin{pmatrix} A_1^T & I_{|V(C)|-r} \end{pmatrix}$, where
\[
\begin{pmatrix} 0 & A_1 \\ A_1^T & 0 \end{pmatrix}
\]
is the adjacency matrix of the connected component $C$ containing $v$. Proposition \ref{duals} tells us that the two matroid components are duals of each other, when they are thought of as matroids on $\{1, \dots , |V(C)|\}$ according to the ordering of the columns in $\begin{pmatrix} I_r & A_1 \end{pmatrix}$ and $\begin{pmatrix} A_1^T & I_{|V(C)|-r} \end{pmatrix}$. \end{proof}

\section{Examples}
\label{secexam}
In this section we mention some counterexamples to assertions related to Theorem \ref{main}.
\subsection{Graphs that are not forests}
Let $C_3$ be the cycle on three vertices, and $P_3$ the path on three vertices. Then
\[
IAS(C_3)=
\begin{pmatrix}
1 & 0 & 0 & 0 & 1 & 1 & 1 & 1 & 1\\
0 & 1 & 0 & 1 & 0 & 1 & 1 & 1 & 1\\
0 & 0 & 1 & 1 & 1 & 0 & 1 & 1 & 1 
\end{pmatrix}.
\]
Let $\kappa$ be the column vector 
\[
\kappa=
\begin{pmatrix}
1\\
0\\
1 
\end{pmatrix}.
\] 
If we add $\kappa$ to each column of $IAS(C_3)$ with a $1$ in the second row, we obtain the matrix 
\[
B =\begin{pmatrix}
1 & 1 & 0 & 1 & 1 & 0 & 0 & 0 & 0\\
0 & 1 & 0 & 1 & 0 & 1 & 1 & 1 & 1\\
0 & 1 & 1 & 0 & 1 & 1 & 0 & 0 & 0 
\end{pmatrix}.
\]
According to Lemma \ref{iso}, $M[IAS(C_3)] \cong M[IA(B)]$. On the other hand, the columns of $B$ are the same as the columns of
\[
IAS(P_3)=
\begin{pmatrix}
1 & 0 & 0 & 0 & 1 & 0 & 1 & 1 & 0\\
0 & 1 & 0 & 1 & 0 & 1 & 1 & 1 & 1\\
0 & 0 & 1 & 0 & 1 & 0 & 0 & 1 & 1 
\end{pmatrix} \text{,}
\]
in a different order. It follows that $M[IAS(C_3)] \cong M[B] \cong M[IAS(P_3)]$. We see that the implication $3 \implies 1$ of Theorem \ref{main} does not hold for graphs that are not forests.

The restricted isotropic matroid $M[IA(C_3)]$ has only one component, so it illustrates the fact that Proposition \ref{comps} does not hold for graphs that are not forests. 

Now, let $C_4$ and $P_4$ be the cycle and path on four vertices. Then
\[
IA(P_4)=
\begin{pmatrix}
1 & 0 & 0 & 0 & 0 & 1 & 0 & 0\\
0 & 1 & 0 & 0 & 1 & 0 & 1 & 0\\
0 & 0 & 1 & 0 & 0 & 1 & 0 & 1\\
0 & 0 & 0 & 1 & 0 & 0 & 1 & 0
\end{pmatrix}.
\]
Let $\kappa_1$ be the first column of $IA(P_4)$, and $\kappa_4$ the fourth column. Let $B$ be the matrix obtained from $IA(P_4)$ by adding $\kappa_1$ to every column with a $1$ in the third row, and $\kappa_4$ to every column with a $1$ in the second row. Then
\[
B=
\begin{pmatrix}
1 & 0 & 1 & 0 & 0 & 0 & 0 & 1\\
0 & 1 & 0 & 0 & 1 & 0 & 1 & 0\\
0 & 0 & 1 & 0 & 0 & 1 & 0 & 1\\
0 & 1 & 0 & 1 & 1 & 0 & 0 & 0
\end{pmatrix}
\]
has the same columns as
\[
IA(C_4)=
\begin{pmatrix}
1 & 0 & 0 & 0 & 0 & 1 & 0 & 1\\
0 & 1 & 0 & 0 & 1 & 0 & 1 & 0\\
0 & 0 & 1 & 0 & 0 & 1 & 0 & 1\\
0 & 0 & 0 & 1 & 1 & 0 & 1 & 0
\end{pmatrix} \text{,}
\]
so Lemma \ref{iso} tells us that $M[IA(P_4)] \cong M[B] \cong M[IA(C_4)]$. We see that the implication $2 \implies 1$ of Theorem \ref{main} does not hold for graphs that are not forests.
\subsection{Strange isomorphisms of isotropic matroids}
\label{strange}

Theorem \ref{main} asserts that if there is an isomorphism between the (restricted) isotropic matroids of two forests, then there is also an isomorphism between the forests themselves. The theorem does not assert that a matroid isomorphism gives rise directly to an isomorphism between the forests. Generally speaking, in fact, a matroid isomorphism need not be directly related to any particular graph isomorphism.

One way to appreciate this fact is to observe that the matrix $IA(P_3)$ has three identical columns, and three other columns that constitute a circuit of size three in $M[IA(P_3)]$. In each case, the three corresponding matroid elements are indistinguishable, up to automorphism. It follows that the automorphism group of $M[IA(P_3)]$ is isomorphic to $S_3 \times S_3$, where $S_3$ is the symmetric group. In contrast, the graph $P_3$ has only two automorphisms. 

Another relevant example is $P_4$. Here is $IAS(P_4)$:
\[
\begin{pmatrix}
1 & 0 & 0 & 0 & 0 & 1 & 0 & 0 & 1 & 1 & 0 & 0 \\
0 & 1 & 0 & 0 & 1 & 0 & 1 & 0 & 1 & 1 & 1 & 0 \\
0 & 0 & 1 & 0 & 0 & 1 & 0 & 1 & 0 & 1 & 1 & 1 \\
0 & 0 & 0 & 1 & 0 & 0 & 1 & 0 & 0 & 0 & 1 & 1 
\end{pmatrix}.
\]
Produce a matrix $B$ from $IAS(P_4)$ in three steps: interchange the first and fourth rows, add the second column to each column with a nonzero entry in the fourth row, and add the third column to each column with a nonzero entry in the first row. That is,
\[
B = \begin{pmatrix}
0 & 0 & 0 & 1 & 0 & 0 & 1 & 0 & 0 & 0 & 1 & 1 \\
1 & 1 & 0 & 0 & 1 & 1 & 1 & 0 & 0 & 0 & 1 & 0 \\
0 & 0 & 1 & 1 & 0 & 1 & 1 & 1 & 0 & 1 & 0 & 0 \\
1 & 0 & 0 & 0 & 0 & 1 & 0 & 0 & 1 & 1 & 0 & 0 
\end{pmatrix}.
\]
Permuting the rows of a matrix has no effect on the matroid represented by the matrix, and the other two steps in the production of $B$ fall under Lemma \ref{iso}, so $B$ represents the matroid $M[IAS(P_4)]$, with the columns of $B$ representing matroid elements in the same order as the columns of $IAS(P_4)$. The matroid elements are listed, in order, in the first row of the matrix displayed below. (Here $V(P_4)=\{v_1,v_2,v_3,v_4\}$, with the vertices indexed in order along the path.) On the other hand, the columns of $B$ are equal to columns of $IAS(P_4)$, as listed in the second row of the matrix displayed below.
\[
\begin{pmatrix}
\phi(v_1) & \phi(v_2) & \phi(v_3) & \phi(v_4) & \chi(v_1) & \chi(v_2) & \chi(v_3) & \chi(v_4) & \psi(v_1) & \psi(v_2) & \psi(v_3) & \psi(v_4) \\
\chi(v_3) & \chi(v_1) & \chi(v_4) & \chi(v_2) & \phi(v_2) & \psi(v_3) & \psi(v_2) & \phi(v_3) & \phi(v_4) & \psi(v_4) & \psi(v_1) & \phi(v_1)
\end{pmatrix}
\]

It follows that the $2 \times 12$ matrix displayed above represents an automorphism $f$ of $M[IAS(P_4)]$. Theorem \ref{main} is satisfied, of course, as $P_4$ is isomorphic to itself. But neither automorphism of $P_4$ appears to be associated in a direct way with $f$.

In the proof of Theorem \ref{main}, the lack of a direct relationship between arbitrary graph and matroid isomorphisms is managed by refining the statement of Theorem \ref{main} to require some agreement between the two types of isomorphisms. See Theorems \ref{treethm} and \ref{closethm} below.

\section{Part 1 of the proof}
\label{twoproof}
In this section we prove the implication $2 \implies 1$ of Theorem \ref{main}.
\begin{lemma}
\label{seraut}
Suppose $v$ is a leaf of a forest $F$, and $w$ is its only neighbor. Then both of the transpositions  $(\phi(v)\chi(w))$ and $(\chi(v)\phi(w))$, in cycle notation, are automorphisms of $M[IA(F)]$.
\end{lemma}
\begin{proof}
The transposition $(\chi(v)\phi(w))$ is an automorphism of $M[IA(F)]$ because the $\chi(v)$ and $\phi(w)$ columns of $IA(F)$ are equal.

The $\phi(v)$ and $\chi(w)$ columns of $IA(F)$ are the only columns with nonzero entries in the $v$ row. Let $\kappa$ be the sum of these two columns. Then $\kappa$ is a column vector with a $0$ in the $v$ row, so according to Lemma \ref{iso}, adding $\kappa$ to the $\phi(v)$ and $\chi(w)$ columns produces another matrix representing the same matroid. Adding $\kappa$ to these columns has the same effect as interchanging them, so the transposition $(\phi(v)\chi(w))$ is an automorphism of $M[IA(F)]$.
\end{proof}

\begin{lemma}
\label{con}
Let $v$ be a leaf of a forest $F$, and $w$ its unique neighbor. If $M$ is the component of $M[IA(F)]$ that contains $\chi(v)$, then $M \setminus \chi(v)$ is isomorphic to the component of $M[IA(F \setminus v)]$ that contains $\phi(w)$. The isomorphism is natural, in that the image of each element $\phi(x)$ or $\chi(x)$ from $M \setminus \chi(v)$ is the element $\phi(x)$ or $\chi(x)$ from $M[IA(F \setminus v)]$.
\end{lemma}
\begin{proof}
The matroid $M$ is represented by the submatrix of $IA(F)$ consisting of the $\phi(x)$ columns such that the distance between $w$ and $x$ is finite and even, and the $\chi(x)$ columns such that the distance between $w$ and $x$ is finite and  odd. Every entry of the $v$ row of this submatrix is $0$, so the $v$ row may be deleted without affecting the matroid represented by the matrix. If we then remove the $\chi(v)$ column, we obtain the submatrix of $IA(F \setminus v)$ that represents the component of $M[IA(F \setminus v)]$ containing $\phi(w)$. 
\end{proof}

\begin{theorem}
\label{treethm}
Suppose $T$ and $T'$ are trees, and there is a matroid isomorphism $f$ between a component of $M[IA(T)]$ and a component of $M[IA(T')]$. Then there is a graph isomorphism $g:T \to T'$, with $g(v)=v'$ whenever $f(\phi(v))=\phi(v')$ or $f(\chi(v))=\chi(v')$.
\end{theorem}
\begin{proof}
Suppose a component of $M[IA(T)]$ is isomorphic to a component of $M[IA(T')]$. According to Proposition \ref{comps}, it follows that $T$ and $T'$ have the same number $n$ of vertices. If $n =1$, the theorem follows immediately.

We proceed using induction on $n>1$. Suppose a component of $M[IA(T)]$ is isomorphic to a component of $M[IA(T')]$. According to Proposition \ref{comps}, each of the matroids $M[IA(T)],M[IA(T')]$ has two components, and the components are duals of each other. It follows that each component of $M[IA(T)]$ is isomorphic to a component of $M[IA(T')]$.

Let $v$ be a leaf of $T$, and $w$ its unique neighbor. Let $M$ be the component of $M[IA(T)]$ that contains the parallel pair $\{\chi(v),\phi(w)\}$. Let $f:M \to M'$ be an isomorphism between $M$ and a component $M'$ of $M[IA(T')]$. Of course it follows that $\{f(\chi(v)),f(\phi(w))\}$ is a parallel pair in $M'$.

According to Proposition \ref{props}, the parallel pair $\{f(\chi(v)),f(\phi(w))\}$ includes $\chi(v')$ for some leaf $v'$ of $T'$. The other element of $\{f(\chi(v)),f(\phi(w))\}$ might be $\chi(v'')$ for some leaf $v'' \neq v'$ which shares the unique neighbor $w'$ of $v'$. If it is, Lemma \ref{seraut} tells us that the transposition $(\chi(v'')\phi(w'))$ defines an automorphism of $M[IA(T')]$, and hence of $M'$. Composing this automorphism with $f$, if necessary, we may assume that $\{f(\chi(v)),f(\phi(w))\} = \{\chi(v'),\phi(w')\}$. Lemma \ref{seraut} also tells us that $(\chi(v')\phi(w'))$ defines an automorphism of $M'$; composing this automorphism with $f$, if necessary, we may assume that $f(\chi(v))=\chi(v')$ and $f(\phi(w))=\phi(w')$. 

Then $f$ defines an isomorphism $g:M \setminus \chi(v) \to M' \setminus \chi(v')$ by restriction. The image of $\phi(w)$ under $g$ is $f(\phi(w))=\phi(w')$. According to Lemma \ref{con}, $g$ defines a matroid isomorphism between a component of $M[IA(T \setminus v)]$ and a component of $M[IA(T' \setminus v')]$. The inductive hypothesis tells us that there is a graph isomorphism $h:T \setminus v \to T' \setminus v'$, with $ h(w)=w'$. Attaching $v$ to $w$ and $v'$ to $w'$, we see that $h$ extends to an isomorphism between $T$ and $T'$.
\end{proof}

Now, suppose $F$ and $F'$ are forests such that $M[IA(F)] \cong M[IA(F')]$. Then the two matroids have the same cardinality, so the two forests have the same number $n$ of vertices. If $n=1$, then of course the two forests are isomorphic.

The argument proceeds by induction on $n>1$. According to Proposition \ref{comps}, each $M[IA]$ matroid has twice as many components as the corresponding forest has connected components. If each matroid has only two components, the two forests are trees and Theorem \ref{treethm} applies. If each matroid has more than two components, choose a component $M$ in $M[IA(F)]$. Under an isomorphism $M[IA(F)] \cong M[IA(F')]$, $M$ corresponds to an isomorphic component of $M[IA(F')]$. Proposition \ref{comps} and Theorem \ref{treethm} tell us that the corresponding connected components of $F$ and $F'$ are isomorphic trees. After removing these isomorphic trees from $F$ and $F'$, and deleting a component isomorphic to $M$ and a component isomorphic to $M^*$ from each of $M[IA(F)]$ and $M[IA(F')]$, we may apply the inductive hypothesis to the remaining submatroids.

\section{Part 2 of the proof}
\label{threeproof}

The proof of the implication $3 \implies 1$ of Theorem \ref{main} is a bit longer than the proof of $2 \implies 1$. 

\begin{lemma}
\label{pendlem}
Let $v$ be a leaf in a forest $F$, let $w$ be the unique neighbor of $v$, and let $\beta_{vw}:M[IAS(F)] \to M[IAS(F)]$ be the permutation 
\[
\beta_{vw} = (\phi(v) \psi(w))(\chi(v) \phi(w))(\psi(v) \chi(w)) \text{,}
\]
in cycle notation. Then $\beta_{vw}$ is an automorphism of the matroid $M[IAS(F)]$.
\end{lemma}
\begin{proof}
Let $\kappa$ be the column vector with a $1$ in the $x$ row whenever $x$ is either $w$ or a neighbor of $w$ other than $v$.

The $v$ entry of $\kappa$ is $0$, so according to Lemma \ref{iso}, a new matrix representing $M[IAS(F)]$ can be obtained by adding $\kappa$ to every column of $IAS(F)$ whose entry in the $v$ row is $1$. The only columns of $IAS(F)$ with nonzero entries in the $v$ row are $\phi(v),\psi(v),\chi(w)$ and $\psi(w)$. For each of the four, the effect of adding $\kappa$ is the same as the effect of applying the permutation $\beta_{vw}$. 

The lemma follows, because the columns of $IAS(F)$ corresponding to $\chi(v)$ and $\phi(w)$ are equal. \end{proof}

Here is a similar observation.

\begin{lemma}
\label{pendlem2}
Let $v$ be a leaf in a forest $F$, and let $w$ be the unique neighbor of $v$. Then the permutation $\gamma_{vw}=(\phi(v) \psi(v))(\chi(w) \psi(w))$ is an automorphism of  $M[IAS(F)]$.
\end{lemma}
\begin{proof} Let $\kappa=\phi(w)$, and apply Lemma \ref{iso} to all the columns of $IAS(F)$ with nonzero entries in the $v$ row. \end{proof}

\begin{definition} \label{tridef} A \emph{triangulation} of a matroid is a partition of the matroid into $3$-element subsets, each of which is either a circuit or a union of two circuits. Two triangulations are \emph{equivalent} if one is the image of the other under an automorphism of the matroid.\end{definition}

The $3$-element subsets included in a triangulation $\mathfrak{T}$ are the \emph{triples} of $\mathfrak{T}$. The \emph{vertex triangulation} of the isotropic matroid of a graph is the triangulation whose triples are vertex triples. In general, an isotropic matroid will have many different triangulations. For instance, if $x$ and $y$ are parallel in $M[IAS(F)]$, then the transposition $(xy)$ is an automorphism of $M[IAS(F)]$, and the image under $(xy)$ of a triangulation $\mathfrak{T}$ may well be different from $\mathfrak{T}$. We call this kind of automorphism a \emph{parallel swap}.

\begin{definition} \label{swapdef} Two triangulations of a matroid are \emph{equivalent through parallel swaps}, or \emph{ps-equivalent}, if one can be obtained from the other through a sequence of parallel swaps. \end{definition}

\begin{lemma}
\label{degtwolem}
Let $T$ be a tree with four or more vertices, and let $\tau$ be a $3$-element circuit of $M[IAS(T)]$, which cannot be changed into a vertex triple using parallel swaps. Then $T$ has a degree-$2$ vertex $z$, with neighbors $x$ and $y$, such that parallel swaps change $\tau$ into either $\{\phi(x),\phi(y),\chi(z)\}$ or $\{\psi(x),\phi(y),\psi(z)\}$. In the latter case, $x$ is a leaf of $T$.
\end{lemma}
\begin{proof}
Suppose $\tau=\{\gamma(x),\delta(y),\epsilon(z)\}$, where $\gamma,\delta, \epsilon \in\{\phi,\chi,\psi\}$ and $x,y,z \in V(T)$. If two of $x,y,z$ are equal, then two elements of $\tau$ come from the corresponding vertex triple. It follows that the third element of $\tau$ is parallel to the third element of the vertex triple, contradicting the hypothesis that $\tau$ cannot be changed into a vertex triple by parallel swaps. Hence $x \neq y \neq z \neq x$. 

Every element of $\tau$ corresponds to a column of $IAS(T)$. As $\tau$ is a circuit in $M[IAS(T)]$, each nonzero entry in one of the three columns is matched by a nonzero entry in precisely one of the other two columns. 

Suppose every two of these three columns share a nonzero entry; then each column has at least two nonzero entries, as the three columns do not have a common nonzero entry, so $\gamma, \delta, \epsilon \in \{\chi, \psi\}$. As $T$ is a tree, $x,y,z$ cannot all be neighbors. Suppose that two pairs of $x,y,z$ are neighbors, and the third pair is not; say $x$ and $y$ are not neighbors. As $T$ is a tree, the neighbors $x$ and $z$ cannot share a neighbor; and the same for $y$ and $z$. If $z$ has a neighbor other than $x$ or $y$, we contradict the fact that each nonzero entry of $\epsilon(z)$ must be matched in $\gamma(x)$ or $\delta(y)$; hence the degree of $z$ is $2$. If $x$ has a neighbor $v$ other than $z$, the nonzero entry of $\gamma(x)$ in the $v$ row must be matched in $\delta(y)$, as it is not matched in $\epsilon(z)$; but then $x,v,y,z,x$ is a closed walk in $T$, an impossibility. The same contradiction arises if $y$ has a neighbor other than $z$, so $V(T)=\{x,y,z\}$, contradicting the hypothesis that $|V(T)| \geq 4$. Suppose only one pair of $x,y,z$ are neighbors; say $x$ is adjacent to $y$, and neither neighbors $z$. Then each of $x,y$ shares a neighbor with $z$; but this is impossible in a tree. We conclude that none of $x,y,z$ are adjacent, and each pair has a shared neighbor. The shared neighbors must be different, as the three columns do not have a common nonzero entry; again, this is impossible in a tree.

Renaming $x,y,z$ if necessary, we may assume that $\gamma(x)$ and $\delta(y)$ do not share a nonzero entry. Then $\epsilon(z)$ has a nonzero entry in every row where $\gamma(x)$ or $\delta(y)$ has a nonzero entry, and nowhere else.

If $\gamma(x)$ has more than two nonzero entries, then $x$ and $z$ share two neighbors, an impossibility in a tree. The same holds for $\delta(y)$, so each of $\gamma(x),\delta(y)$ has one or two nonzero entries.

Suppose $\gamma(x)$ has two nonzero entries. If $\gamma=\chi$ then either $x$ and $z$ share two neighbors, or $x$ and $z$ are neighbors that share a neighbor. Neither situation is possible in a tree, so $\gamma=\psi$. As $\psi(x)$ has only two nonzero entries, $x$ is a leaf. As $\epsilon(z)$ shares the $1$ entry of $\psi(x)$ in the $x$ row, $x$ and $z$ are neighbors. Hence the entry of $\psi(x)$ in the $z$ row is $1$; $\epsilon(z)$ shares this entry, so $\epsilon = \psi$. If $\delta(y)$ were to have two nonzero entries, then as $\gamma(x)$ and $\delta(y)$ do not share a nonzero entry, neither nonzero entry of $\delta(y)$ would be in the $z$ row. Both nonzero entries of $\delta(y)$ would be matched in $\epsilon(z)=\psi(z)$, so $y$ and $z$ would share two neighbors, an impossibility. Hence $\delta(y)$ has only one nonzero entry, so the $\delta(y)$ column of $IAS(T)$ is equal to the $\phi(y')$ column for some vertex $y'$. As the nonzero entry of $\phi(y')$ is matched in $\epsilon(z)=\psi(z)$, $y'$ is a neighbor of $z$. A parallel swap changes $\tau$ into $\{\psi(x),\phi(y'),\psi(z)\}$ with $x$ a leaf, $z$ the only neighbor of $x$, and $y'$ the only other neighbor of $z$.

If $\delta(y)$ has two nonzero entries, interchange the names of $x$ and $y$ and apply the argument of the preceding paragraph.

The only remaining possibility is that each of $\gamma(x),\delta(y)$ has only one nonzero entry. Using parallel swaps, we can change $\tau$ into $\tau'=\{\phi(x'),\phi(y'), \epsilon(z)\}$. If $\epsilon(z)=\psi(z)$ then one of $x',y'$ is $z$, so $\tau'$ contains two elements of the vertex triple of $z$; a parallel swap then changes $\tau'$ into the vertex triple of $z$, contrary to hypothesis. It follows that $\epsilon(z)=\chi(z)$, and hence $x'$ and $y'$ are the only neighbors of $z$. \end{proof}

As mentioned in the introduction, the isotropic matroids of general graphs detect only equivalence under local complementation, not equivalence under isomorphism. A structural reflection of the special properties of isotropic matroids of forests is the fact that triangulations not ps-equivalent to vertex triangulations are not unusual for general graphs, as discussed in \cite{Tnewnew}, but they are quite rare for forests.

\begin{theorem}
\label{bigthm}
Let $T$ be a tree whose isotropic matroid has a triangulation that is not ps-equivalent to the vertex triangulation. Then $T$ is isomorphic to the four-vertex path, $P_4$.
\end{theorem}
\begin{proof}
It is easy to see that if $T$ is a tree on one or two vertices, $M[IAS(T)]$ does not have a triangulation that is not ps-equivalent to the vertex triangulation.

Up to isomorphism, there is only one tree on three vertices, the path $P_3$. If there were a triangulation of $M[IAS(P_3)]$ not ps-equivalent to the vertex triangulation, it would include a triple that meets all three vertex triples. Therefore, the triangulation would not include any vertex triple; every triple would meet all three vertex triples. If $v$ is one of the leaves of $P_3$ and $w$ is the central vertex, then one triple of the triangulation would include $\phi(v)$ and one of $\chi(w),\psi(w)$, because these are the only elements available to cancel the nonzero entry of $\phi(v)$ in the $v$ row. For the same reason, a triple would include $\psi(v)$ and one of $\chi(w),\psi(w)$. It follows that the triple containing the remaining element of the vertex triple of $v$, $\chi(v)$, would also contain the remaining element of the vertex triple of $w$, $\phi(w)$. But this is impossible: the matrix $IAS(P_3)$ does not have a column of zeroes, so there is no way to complete this triple. Hence $M[IAS(P_3)]$ does not have a triangulation that is not ps-equivalent to the vertex triangulation.

Suppose $T$ has four or more vertices and we are given a triangulation $\mathfrak{T}$ of $M[IAS(T)]$, which is not ps-equivalent to the vertex triangulation, and has the smallest possible number of triples that are not vertex triples. Let $\tau$ be a non-vertex triple included in $\mathfrak{T}$.

Suppose $\tau$ can be changed into a vertex triple using parallel swaps. According to Proposition \ref{props}, no $\psi$ element of $M[IAS(T)]$ is parallel to any other element of $M[IAS(T)]$. Hence there is a $v \in V(T)$ with $\psi(v) \in \tau$, and parallel swaps can be used to change $\tau$ into the vertex triple of $v$. The element(s) to be swapped into $\tau$, $\phi(v)$ and/or $\chi(v)$, must appear in one or two other triples of $\mathfrak{T}$. Neither of these other triples could be a vertex triple, as $\psi(v)$ appears in $\tau$. Therefore the parallel swap(s) that change $\tau$ into the vertex triple of $v$ also produce a triangulation ps-equivalent to $\mathfrak{T}$, which has a smaller number of non-vertex triples. This contradicts our choice of $\mathfrak{T}$.

According to Lemma \ref{degtwolem}, we may perform parallel swaps to change $\tau$ into $\tau' = \{\phi(x), \phi(y), \chi(z)\}$ or $\widetilde \tau=\{\psi(x),\phi(y),\psi(z)\}$, where $x,y,z$ are three distinct vertices of $T$, and $x$ and $y$ are the only neighbors of $z$. These parallel swaps transform $\mathfrak{T}$ into a triangulation $\mathfrak{T}'$ or $\widetilde {\mathfrak{T}}$. 

If the second of these possibilities occurs, then Lemma \ref{degtwolem} tells us that $x$ is a leaf. The triple $\widetilde \tau' \in \widetilde {\mathfrak{T}}$ that contains $\phi(x)$ must also contain $\chi(z)$, as these are the only two columns of $IAS(T)$ that are not included in $\widetilde \tau$ and have nonzero entries in the $x$ row. In order to complete a dependent set, the third element of $\widetilde \tau'$ must have only one nonzero entry, in the $y$ row. Therefore a parallel swap changes $\widetilde \tau'$ into $\{\phi(x),\phi(y),\chi(z)\}$. 

We conclude that no matter which of the two possibilities mentioned in the paragraph before last occurs, parallel swaps can be used to change $\mathfrak{T}$ into a triangulation $\mathfrak{T}'$ that includes $\tau'=\{\phi(x),\phi(y),\chi(z)\}$, where $x$ and $y$ are the only neighbors of $z$. The rest of the argument is focused on $\tau'$ and $\mathfrak{T}'$; $\tau$ and $\widetilde \tau$ will not be mentioned again.

Let $\tau''$ be the triple of $\mathfrak{T}'$ that includes $\phi(z)$. One of the other elements of $\tau''$ must correspond to a column of $IAS(T)$ with a nonzero entry in the $z$ row. As $x$ and $y$ are the only neighbors of $z$, this other element must be one of the following: $\chi(x),\psi(x),\chi(y),\psi(y),\psi(z)$. If $\psi(z) \in \tau''$ then the third element of $\tau''$ is parallel to $\chi(z)$; it cannot equal $\chi(z)$ because $\chi(z) \in \tau'$. This parallel of $\chi(z)$ must correspond to a column of $IAS(T)$ with the same two nonzero entries as $\chi(z)$. But this is impossible, as the corresponding vertex of $T$ would either neighbor both $x$ and $y$, or equal one of $x,y$ and neighbor the other; either way, the fact that $x$ and $y$ both neighbor $z$ would imply that $T$ contains a circuit. Therefore $\psi(z) \notin \tau''$. Interchanging the names of $x$ and $y$ if necessary, we may presume that $\tau''$ contains $\chi(x)$ or $\psi(x)$.

Case 1. Suppose $\tau''$ contains $\chi(x)$ along with $\phi(z)$. Then $\tau''=\{\chi(x),\phi(z),\kappa \}$, where $\kappa$ corresponds to a column of $IAS(T)$ with nonzero entries in all the rows corresponding to neighbors of $x$ other than $z$, and nowhere else. If $\kappa$ has more than one nonzero entry, the vertex that contributes this column to $IAS(T)$ is either a neighbor of $x$ that shares a neighbor with $x$, or a vertex that shares two neighbors with $x$. Either way, this vertex lies on a closed walk with $x$, an impossibility as $T$ is a tree. Therefore $\kappa$ has only one nonzero entry. We conclude that $x$ is of degree two; one neighbor is $z$ and the other, $w$ say, has $\phi(w)$ parallel to $\kappa$. In sum: $T$ contains a path $w,x,z,y$; $x$ and $z$ are of degree $2$ in $T$; and $\mathfrak{T}'$ includes $\tau' = \{\phi(x), \phi(y), \chi(z)\}$ and $\tau''=\{\phi(z),\chi(x),\kappa\}$, where $\kappa$ is parallel to $\phi(w)$.

The triangulation $\mathfrak{T}'$ must also include triple(s) that contain $\psi(x)$ and $\psi(z)$. If $\psi(x)$ and $\psi(z)$ were to appear together in a triple, the third element of that triple would have nonzero entries in the $w$ and $y$ rows; but then $w$ and $y$ would share a neighbor, or neighbor each other, and either way the path $w,x,z,y$ would be contained in a closed walk in $T$. This is impossible, so $\mathfrak{T}'$ includes a triple $\tau'''$ that contains $\psi(x)$ but not $\psi(z)$. Considering that the elements of $\tau'$ and $\tau''$ are not available for $\tau'''$, the only way for $\tau'''$ to match the entry of $\psi(x)$ in the $x$ row is for $\tau'''$ to contain $\chi(w)$ or $\psi(w)$. Similarly, to match the entry of $\psi(x)$ in the $z$ row, $\tau'''$ must contain $\chi(y)$ or $\psi(y)$. These three elements -- $\chi(w)$ or $\psi(w)$, $\psi(x)$, and $\chi(y)$ or $\psi(y)$ -- will not provide a dependent triple $\tau'''$ if $w$ has any neighbor other than $x$, or $y$ has any neighbor other than $z$. We deduce that $w,x,y$ and $z$ are the only vertices in $T$, and $T$ is isomorphic to $P_4$.

Case 2. Suppose $\tau''$ contains $\psi(x)$ along with $\phi(z)$. Then $\tau''=\{\psi(x),\phi(z),\kappa \}$ where $\kappa$ corresponds to a column of $IAS(T)$ with nonzero entries in the $x$ row, in all the rows corresponding to neighbors of $x$ other than $z$, and nowhere else. If $\kappa$ has more than two nonzero entries, the vertex that contributes the $\kappa$ column to $IAS(T)$ is a neighbor of $x$ that shares a neighbor with $x$. This is impossible, as $T$ is a tree. If $\kappa$ is a $\chi$ column with two nonzero entries, the same contradiction arises. 

Suppose $\kappa$ has only one nonzero entry. This entry must be in the $x$ row, so the only neighbor of $x$ is $z$. As $\phi(x) \in \tau'$ and $\kappa \in \tau''$, $\kappa \neq \phi(x)$. Hence $\kappa=\chi(w)$ for some neighbor $w$ of $x$. But this is impossible; the only neighbor of $x$ is $z$, and the entry of $\chi(z)$ in the $y$ row is $1$, not $0$.

The only remaining possibility is that there is a vertex $w$ such that $\kappa=\psi(w)$ has two nonzero entries, in the $x$ and $w$ rows. Then $w \neq x $, because $x$ neighbors $z$ and the $z$ coordinate of $\kappa=\psi(w)$ is $0$, and $w \notin \{y,z\}$, because both $\psi(y)$ and $\psi(z)$ have entries equal to $1$ in the $y$ and $z$ rows. It follows that $w$ is a leaf, and $x$ is its unique neighbor. Moreover, $\psi(x)$ has nonzero entries only in the $w,x$ and $z$ rows, so $w$ and $z$ are the only neighbors of $x$.

We conclude that $T$ contains a path $w,x,z,y$, with $w$ a leaf and $x,z$ of degree two. Two triples of $\mathfrak{T}'$ are $\tau'=\{\phi(x),\phi(y),\chi(z)\}$ and $\tau''=\{\psi(x),\phi(z),\psi(w)\}$. Some other triple $\tau''' \in \mathfrak{T}'$ must contain $\chi(x)$, and $\tau'''$ must contain a second element with a nonzero entry in the $w$ row; $\tau''$ contains $\psi(x)$ and $\psi(w)$, so this second element of $\tau'''$ can only be $\phi(w)$. The third element of $\tau'''$ must correspond to a column of $IAS(T)$ whose only nonzero entry is in the $z$ row, but it cannot be $\phi(z)$, because $\phi(z) \in \tau''$. Therefore $z$ neighbors a leaf of $T$. As $x$ is not a leaf, and $y$ is the only other neighbor of $z$, it follows that $y$ is a leaf, and hence that $T$ is isomorphic to $P_4$. \end{proof}

\begin{theorem}
\label{bigthm2}
If $T$ is a tree, then every triangulation of $M[IAS(T)]$ is equivalent to the vertex triangulation.
\end{theorem}
\begin{proof}
According to Theorem \ref{bigthm}, we need only consider the special case $T=P_4$. Let $V(T)=\{v_1,v_2,v_3,v_4\}$, with the vertices indexed in order along the path, and let $\mathfrak{T}$ be a triangulation of $M[IAS(T)]$.

Let $\tau_1 \in \mathfrak{T}$ be the vertex triple that contains $\phi(v_1)$. Then $\tau_1$ contains one other element corresponding to a column of $IAS(T)$ with a nonzero entry in the $v_1$ row. This other element must be $\psi(v_1),\chi(v_2)$ or $\psi(v_2)$.

Case 1: $\psi(v_1)\in \tau_1$. As $\tau_1$ is dependent, its third element is parallel to $\chi(v_1)$. Therefore $\mathfrak{T}$ is ps-equivalent to a triangulation $\mathfrak{T}'$, which includes the vertex triple of $v_1$. The triple $\tau_2 \in \mathfrak{T}'$ that includes $\chi(v_2)$ must also include $\psi(v_2)$, as these are the only remaining columns of $IAS(T)$ with nonzero entries in the $v_1$ row. The third element of $\tau_2$ is parallel to $\phi(v_2)$; it cannot be $\chi(v_1)$, as $\mathfrak{T}'$ includes the vertex triple of $v_1$, so it must equal $\phi(v_2)$. Thus $\mathfrak{T}'$ includes the vertex triples of both $v_1$ and $v_2$. The two remaining columns of $IAS(T)$ with nonzero entries in the $v_2$ row are $\chi(v_3)$ and $\psi(v_3)$; they must appear in the same triple of $\mathfrak{T}'$, and the third element of this triple must be parallel to $\phi(v_3)$. A parallel swap (if needed) then transforms $\mathfrak{T}'$ into a triangulation that contains the vertex triples of $v_1,v_2$ and $v_3$; this must be the vertex triangulation.

Case 2: $\chi(v_2)\in \tau_1$. The third element of $\tau_1$ must be parallel to $\chi(v_4)$, so $\mathfrak{T}$ is ps-equivalent to a triangulation $\mathfrak{T}'$ with $\tau'_1=\{\phi(v_1),\chi(v_2),\chi(v_4)\} \in \mathfrak{T}'$. Let $f$ be the automorphism of $M[IAS(T)]$ discussed in Section \ref{strange}. Then $f(\mathfrak{T}')$ includes $f(\tau'_1)$, which is the vertex triple of $v_3$. The two remaining columns of $IAS(T)$ with nonzero entries in the $v_4$ row are $\phi(v_4)$ and $\psi(v_4)$; they must appear together in a triple of $f(\mathfrak{T}')$, and this triple can only be the vertex triple of $v_4$. The two remaining columns of $IAS(T)$ with nonzero entries in the $v_3$ row are $\chi(v_2)$ and $\psi(v_2)$; they must be included in a single triple of $f(\mathfrak{T}')$. After a parallel swap (if needed) this triple is the vertex triple of $v_2$, so $f(\mathfrak{T}')$ is ps-equivalent to a triangulation that includes the vertex triples of $v_2,v_3$ and $v_4$. This can only be the vertex triangulation. 

Case 3: $\psi(v_2)\in \tau_1$. The third element of $\tau_1$ then corresponds to a column of $IAS(T)$ with precisely two nonzero entries, one in the $v_2$ row and one in the $v_3$ row. No such column exists.
\end{proof}
We should mention that Theorem \ref{bigthm2} extends to arbitrary graphs. See \cite[Prop.\ 29]{Tnewnew} for details.

The next lemma is useful for induction.
\begin{lemma}
\label{treeminor}
Let $v$ be a vertex of a forest $F$, and let $F \setminus v$ be the forest obtained from $F$ by removing $v$. Then
\[
(M[IAS(F)] \setminus \{\chi(v),\psi(v)\}) / \phi(v) \cong M[IAS(F \setminus  v)].
\]
The isomorphism is natural, in that the image of each element $\phi(x),\chi(x)$ or $\psi(x)$ from $(M[IAS(F)]-\{\chi(v),\psi(v)\}) / \phi(v)$ is the element $\phi(x),\chi(x)$ or $\psi(x)$ from $M[IAS(F \setminus  v)]$.
\end{lemma}
\begin{proof}
It is obvious that the matrix $IAS(F \setminus v)$ is the same as the submatrix of $IAS(F)$ obtained by deleting the $v$ row and the $\phi(v),\chi(v)$ and $\psi(v)$ columns. As the only nonzero entry of the $\phi(v)$ column occurs in the $v$ row, the lemma follows from Proposition \ref{contr}.
\end{proof}

One last lemma will be useful in the proof of the implication $3 \implies 1$ from Theorem \ref{main}. Suppose $F$ and $F'$ are forests, and $f$ is an isomorphism between $M[IAS(F)]$ and $M[IAS(F')]$. If $w$ is a vertex of $F$ that is adjacent to a leaf $v$, then $\phi(w)$ and $\chi(v)$ are parallel in $M[IAS(F)]$, so their images under $f$ must be parallel in $M[IAS(F')]$. It follows from Proposition \ref{props} that either $f(\phi(w))= \phi(w')$ for some $w' \in V(F')$ that is adjacent to a leaf, or $f(\phi(w))= \chi(v')$ for some leaf $v' \in V(F')$.

\begin{definition}
In this situation the \emph{leaf index} $i(f)$ is the number of vertices $w \in V(F)$ such that $w$ is adjacent to a leaf, and $f(\phi(w))= \chi(v')$ for some leaf $v' \in V(F')$.
\end{definition}
\begin{lemma}
\label{indexlem}
Suppose $f:M[IAS(F)] \to M[IAS(F')]$ is an isomorphism, which maps each vertex triple of $F$ to a vertex triple of $F'$. Then there is an isomorphism $f':M[IAS(F)] \to M[IAS(F')]$ such that (a) $f'$ maps each vertex triple of $F$ to a vertex triple of $F'$; (b) $i(f')=0$; and (c) whenever $x \in V(F)$ and $x' \in V(F')$ have $f(\phi(x))=\phi(x')$, it is also true that $f'(\phi(x))=\phi(x')$.
\end{lemma}
\begin{proof}
If $i(f)=0$, then $f'=f$ satisfies the lemma.

Suppose $i(f)>0$. Then there is a vertex $w \in V(F)$ such that $w$ neighbors a leaf of $F$, and $f(\phi(w))=\chi(v')$ for some leaf $v'$ of $F'$. Let $w'$ be the unique neighbor of  $v'$ in $F'$. As $f$ is an isomorphism, and $\chi(v')$ is parallel to $\phi(w')$, $f^{-1}(\phi(w'))$ is a parallel of $\phi(w)$. That is, there is a leaf $v$ that neighbors $w$ and has $f(\chi(v))=\phi(w')$. According to Lemma \ref{pendlem}, 
\[
\beta_{vw} = (\phi(v) \psi(w))(\chi(v) \phi(w))(\psi(v) \chi(w)) 
\]
is an automorphism of $M[IAS(F)]$, so the composition $f' = f \circ \beta_{vw}:M[IAS(F)] \to M[IAS(F')]$ is an isomorphism.

The automorphism $\beta_{vw}$ is the identity map outside the vertex triples of $v$ and $w$, and it maps either of these two vertex triples to the other. As $f$ maps vertex triples of $F$ to vertex triples of $F'$, it follows that $f'$ also maps vertex triples of $F$ to vertex triples of $F'$. As $f'$ agrees with $f$ outside the vertex triples of $v$ and $w$, and $f'(\phi(w)) = f(\beta_{vw}(\phi(w))) =f(\chi(v))= \phi(w')$, $i(f')=i(f)-1.$ Notice that $f(\phi(v))$ cannot be a $\phi$ element of $F'$, as it is an element of the same vertex triple as $f(\chi(v))=\phi(w')$. Therefore $f'$ has the property that whenever $x \in V(F)$ and $x' \in V(F')$ have $f(\phi(x))=\phi(x')$, it is also true that $f'(\phi(x))=\phi(x')$. 

Repeating this argument $i(f)$ times, we will find an $f'$ that satisfies the lemma. \end{proof}

We are now ready to verify the implication $3 \implies 1$ of Theorem \ref{main}. The proof uses induction on $n = |V(F)|$. As mentioned at the end of Section \ref{secexam}, we do not prove $3 \implies 1$ as stated in Theorem \ref{main}, because that statement does not give us a sufficiently precise inductive hypothesis. Instead we prove a more detailed theorem: 

\begin{theorem}
\label{closethm}
Suppose $F$ and $F'$ are forests, and $M[IAS(F)] \cong M[IAS(F')]$. Then:
\begin{enumerate}
    \item There is an isomorphism $f:M[IAS(F)] \to M[IAS(F')]$, under which the image of the vertex triangulation of $M[IAS(F)]$ is the vertex triangulation of $M[IAS(F')]$.
    \item For any  $f$ that satisfies part 1, there is a graph isomorphism $g:F \to F'$ with $g(x)=x'$ whenever $x \in V(F)$ and $x' \in V(F')$ have $f(\phi(x))=\phi(x')$.
\end{enumerate}
\end{theorem}
\begin{proof} Let $F$ and $F'$ be forests, with an isomorphism $M[IAS(F)] \cong M[IAS(F')]$. Then the matroids have the same cardinality, so $F$ and $F'$ have the same number $n$ of vertices. 

If $n =1$, there are two matroid isomorphisms $f,f':M[IAS(F)] \to M[IAS(F')]$. One of $f,f'$ matches elements according to their $\phi,\chi,\psi$ designations, and the other of $f,f'$ matches the $\phi$ element of one matroid to the $\psi$ element of the other. (The $\chi$ elements must be matched to each other by any isomorphism, as they are the only loops in the matroids.) Either of the isomorphisms $f,f'$ satisfies the statement, together with the unique graph isomorphism $g:F \to F'$.

The argument proceeds using induction on $n>1$. Suppose the matroids $M[IAS(F)]$ and $M[IAS(F')]$ have more than one component. If $M[IAS(F)]$ has a component $M$ with more than two elements, then $M$ corresponds to a connected component $C$ of $F$ with $1<|V(C)|<n$. The image of $M$ under an isomorphism $M[IAS(F)] \cong M[IAS(F')]$ is a component $M'$  of $M[IAS(F')]$, which corresponds to a connected component $C'$ of $F'$. The theorem holds for $F$ and $F'$ because the inductive hypothesis applies separately to $C$ and $C'$, on the one hand, and $F \setminus C$ and $F' \setminus C'$, on the other hand. If $M[IAS(F)]$ has no component with more than two elements, then all the vertices of $F$ and $F'$ are isolated. It is easy to see that the theorem holds in this trivial case.

Suppose now that $M[IAS(F)]$ and $M[IAS(F')]$ are isomorphic matroids with only one component. Then $F$ and $F'$ are both connected. The image of the vertex triangulation of $M[IAS(F)]$ under an isomorphism $M[IAS(F)] \cong M[IAS(F')]$ is a triangulation of $M[IAS(F')]$. According to Theorem \ref{bigthm2}, this triangulation is equivalent to the vertex triangulation of $M[IAS(F')]$. Therefore, we can compose an automorphism of $M[IAS(F')]$ with the original isomorphism $M[IAS(F)] \cong M[IAS(F')]$ to obtain an isomorphism $f:M[IAS(F)] \to M[IAS(F')]$ that satisfies part 1 of the statement.

Now, let $f:M[IAS(F)] \to M[IAS(F')]$ be any isomorphism that satisfies part 1 of the statement. We need to verify that part 2 of the statement is satisfied by some isomorphism $g:F \to F'$. According to Lemma \ref{indexlem}, there is an isomorphism $f':M[IAS(F)] \to M[IAS(F')]$ such that $i(f')=0$ and any isomorphism $g:F \to F'$ that satisfies part 2 for $f'$ will also satisfy part 2 for $f$. 

Let $v$ be a leaf of $F$, with unique neighbor $w$. As $i(f')=0$, $f'(\phi(w))=\phi(w')$ for some $w'\in V(F')$ that neighbors a leaf. Then $\chi(v)$ is parallel to $\phi(w)$, so $f'(\chi(v))$ is parallel to $f'(\phi(w))=\phi(w')$; hence $f'(\chi(v)) = \chi(v')$ for some leaf $v'$ adjacent to $w'$. As $f'$ maps vertex triples to vertex triples, $f'(\phi(v))$ is either $\phi(v')$ or $\psi(v')$.

If $f'(\phi(v))=\phi(v')$, let $f''=f'$. If $f'(\phi(v))=\psi(v')$, let $f''=\gamma_{v'w'} \circ f'$, where $\gamma_{v'w'}=(\phi(v') \psi(v'))(\chi(w') \psi(w'))$ is the automorphism of $M[IAS(F')]$ mentioned in Lemma \ref{pendlem2}. Then $f'':M[IAS(F)] \to M[IAS(F')]$ is an isomorphism that maps vertex triples to vertex triples and has these properties: $f''(\phi(x))=\phi(x')$ whenever $f(\phi(x))=\phi(x')$, $i(f'')=0$, $f''(\phi(v))=\phi(v')$, $f''(\chi(v))=\chi(v')$, and $f''(\phi(w))=\phi(w')$. The first of these properties implies that any isomorphism $g:F \to F'$ that satisfies part 2 of the statement for $f''$ will also satisfy part 2 for $f$.

As $f''(\phi(v))=\phi(v')$, $f''(\chi(v))=\chi(v')$, and $f''$ maps vertex triples to vertex triples, it must be that $f''(\psi(v))=\psi(v')$. Therefore $f''$ defines an isomorphism
\[
f''': (M[IAS(F)] \setminus \{\chi(v),\psi(v)\})/\phi(v) \to (M[IAS(F')] \setminus \{\chi(v'),\psi(v')\})/\phi(v').
\]
According to Lemma \ref{treeminor}, it follows that $f'''$ defines an isomorphism $f'''':M[IAS(F \setminus v)] \to M[IAS(F' \setminus v')]$. The isomorphism $f''''$ inherits the following properties from $f''$: $f''''$ maps vertex triples to vertex triples; $f''''(\phi(x))=\phi(x')$ whenever $f(\phi(x))=\phi(x')$; $i(f'''')=0$; and $f''''(\phi(w))=\phi(w')$. Applying the inductive hypothesis to $f''''$, we deduce that there is a graph isomorphism $g_0:F \setminus v \to F' \setminus v'$ with $g_0(x)=x$ whenever $f''''(\phi(x))=\phi(x')$. In particular, $g_0(w)=w'$. The inductive step is completed by noting that $g_0$ can be extended to an isomorphism $g:F \to F'$ by defining $g(v)=v'$. This extended isomorphism satisfies part 2 of the statement for $f''$, and hence also for $f$. \end{proof}

\end{document}